\documentclass{article}
\usepackage{algorithm}
\usepackage{amsmath,amsthm}
\usepackage{algpseudocode}
\usepackage{threeparttable}
\usepackage{paralist}
\usepackage{graphics} 
\usepackage{epsfig} 
\usepackage{graphicx}  
\usepackage{epstopdf}
\usepackage[colorlinks=true]{hyperref}
\hypersetup{urlcolor=blue, citecolor=red}

\newtheorem{theorem}{Theorem}[section]

\newtheorem{proposition}{Proposition}

\theoremstyle{definition}

\newtheorem{remark}{Remark}

\begin{document}
\title{A non-iterative algorithm for \\generalized Pig games}

\author{Fabi\'an Crocce\footnote{Centro de Matem\'atica, Facultad de Ciencias, Universidad de la Rep\'ublica,
Igu\'a 4225, 111400, Montevideo, Uruguay,
(fcrocce@gmail.com).}   
\quad
and\quad Ernesto Mordecki\footnote{Centro de Matem\'atica, Facultad de Ciencias, Universidad de la Rep\'ublica,
Igu\'a 4225, 111400, Montevideo, Uruguay
(mordecki@cmat.edu.uy).} }

\maketitle
\begin{abstract}
We provide a polynomial algorithm to find the value and an optimal strategy for a generalization of the Pig game.
Modeled as a competitive Markov decision process, the corresponding  Bellman equations can be decoupled 
leading to  systems of two non-linear equations with two unknowns. In this way we avoid the classical iterative 
approaches. A simple complexity analysis reveals that the algorithm requires 
$O(\mathbf{s}\log\mathbf{s})$ steps, where  $\mathbf{s}$ is the number of states of the game.
The classical Pig and the Piglet (a simple variant of the Pig played with a coin)
are examined in detail. 
\end{abstract}



\section{Introduction}\label{section:intro}
%

Pig is a popular competitive dice game of the family of jeopardy games. 
It is a turn based game with two players. 
Let us describe first the turn of one player, that can be considered as a solitaire variant of the game.
The player rolls a die. 
If he gets a one (the \emph{pig}) he looses the turn,
ending it with zero points.
If he gets a number different from one (say a 4), he scores this result and faces two options:
(i) to bank his points, ending the turn with his first obtained points (4), or 
(ii) to roll the die again. 
The conditions of this second roll are: if he gets a one he looses the score obtained 
in the first roll (his turn account falls from 4 to 0) and the turn is finished. 
If he gets a number different from one (say a 5) he adds this result to the first one, (increasing to 9) 
and has again to decide whether to roll or to hold. 
If he rolls and gets a one, his account goes to 0 (from 9 to 0)
and if he gets a number different from one (say a 3) this number is added to his score (increasing to 12).
The turn continues under the same conditions, until the player decides to hold, or gets a one.
%
%
This single turn can be  modeled as a stochastic optimization problem, 
as after each roll with a result different from one, the player has to decide whether to roll the die, 
possibly increasing his points
--risking to loose these obtained points-- or to stop rolling, ending the turn with these points as reward. 
This turn, considered as a solitaire game, has an optimal rule: to maximize the expected score, 
the player should roll until the first time he accumulates 20 points or more.
The maximum expected score is $8.1418$ points (Roters \cite{Roters}).


We now describe the full competitive Pig game. 
There are two players, who play by turns as described above. 
There is a general account for each player, where they bank the points obtained 
after deciding to hold (ending the  turn), passing the die to the other player.
(These general accounts do not decrease.)
The first player in scoring $100$ or more points wins the game. 


Piglet is a simple version of Pig played with a coin instead of a die. 
It was introduced by  Neller and Presser \cite{Neller}.
In the Piglet, the winner is the first player to obtain 10 heads. 
During his turn, 
each player repeatedly flips a coin until either a tail is obtained or he decides to hold and score the number of consecutive flipped heads. 
Beside its simplicity, it is interesting to note that this game is a particular case of
a \emph{simple stochastic game} (SSG), as introduced by Condon \cite{c1}. 
The solitaire variant of the piglet game has threshold one and maximum expected reward $1/2$.


The \emph{generalized pig game} (GPG) is an abstract form of the previous games, 
played with a ``special die'' with faces $0,\ldots ,n$ and probabilities $p_0,\ldots,p_n$ $(p_0+\cdots+p_n=1)$ respectively.
The first of two players in scoring $N$ or more points wins the game.
Similar to the Pig and Piglet, if in his turn a player rolls the die and obtains a zero, 
the turn ends without increasing the general score, 
whereas if the turn ends by the player's decision, 
the turn score (sum of the outcomes within the turn) is added to his general score.
The Pig is a GPG with $n=6$, $p_0=1/6$, $p_1=0$, $p_i=1/6\ (i=2,\dots,6)$, and $N=100$. 
In the case of the Piglet the parameters are $n=2$, $p_0=p_1=1/2$, and $N=10$. 


Solitaire variants of the Pig game were studied by Roters \cite{Roters} and Roters and Haigh \cite{Haigh}.
The first solution for the competitive Pig game was obtained by Neller and Presser \cite{Neller}, 
with the help of a value iteration algorithm.
Tijms \cite{Tijms} considers also the competitive pig game and a simultaneous-decision-taking variant, 
named Hog (see also \cite{Tijms1}), 
and Louchard \cite{louchard} examines some optimal strategies for the solitaire and competitive variants of the game. 
In these papers, no complexity analysis of the proposed algorithms is presented.


From a theoretical point of view, 
the GPG is a \emph{zero-sum turn-based Competitive Markov Decision Process},
in the terminology of Filar and Vrieze \cite{Filar}, 
or a \emph{two-player zero-sum turn-based stochastic game} (labelled as 2TBSG following Hansen et al. \cite{hansen}), in the terminology of Shapley \cite{Shapley}.
An important theoretical difference between turn-based games and the classic simultaneously-taken actions ones, 
initiated in the work by von Neumann and Morgenstern \cite{mvn}, 
is that simultaneity usually implies the need of randomized strategies to obtain optimality,
whereas in the framework of turn-based games it is usually possible to find deterministic optimal strategies.


Regarding solutions (i.e. finding the value and an optimal strategy),
Vrieze et al. \cite{vet} provide an algorithm to solve a general 2TBSG.
The finiteness of this algorithm has been established, 
and the \emph{order field property} (the solution of the game belongs to the same algebraic field as the data) 
follows from their analysis (see also Theorem 6.3.3 in \cite{Filar}).


As was mentioned our game can be seen to be a Simple Stochastic Game (SSG), 
and therefore all the theoretical results for this class of games apply.
In particular it is known that, for stochastic mean-payoff games, 
both players have optimal pure stationary strategies (Liggett and Lipman \cite{ll}),
and this result provides the corresponding result for the transient game we consider, 
as it can be transformed into a stochastic mean-payoff game by
adding an extra action with reward one that loops at every absorbing state in which player one is the winner, 
while all other actions have null rewards.
Furthermore, in this framework, 
a finite algorithm always exists simply because the number of strategies composed by pure stationary actions is finite. 
Since the number of these strategies is exponential in the number 
$\mathbf{a}$ of the states of the game corresponding to of the players  
(if each of these states has two possible actions, there are $2^\mathbf{a}$ strategies), 
this finite search algorithm is exponential.
For more details about SSGs see \cite{c1} and \cite{c2}.


Hoffman and Karp \cite{hk} provide a strategy iteration algorithm to solve non-terminating stochastic games.
Strategy iteration algorithms are analyzed by several authors. 
Tripathi et al. \cite{tvk} provide a bound of $O(2^s/s)$ for its complexity in the case of SSGs. 
Another approach consists in the reduction of the SSG to linear programming type problems. 
In this direction Halman \cite{halman} provides the bound $e^{O(\sqrt{\mathbf{s} \log \mathbf{s}})}$ for the expected complexity, 
where $\mathbf{s}$ is the number of states of the game,
based on Matou\v{s}ek et al. \cite{matousek}.
Condon \cite{c1} proves that solving a general SSG belongs to a certain complexity class 
(more precisely NP$\,\cap\,$co-NP,  the details in \cite{c1}). 
Condon \cite{c2} also analyzes several other algorithms for SSGs, 
including modifications of the Hoffman-Karp algorithm, 
explains why certain naive algorithms do not provide the correct answer, 
and establishes a quadratic programming algorithm to solve SSGs. 
Within other more recent proposals, we can mention one algorithm based on permutations
of the random nodes of a SSG (Gimbert and Horn \cite{gh}),  
and another one by Ibsen-Jensen and Miltersen \cite{im}, 
that combining value iteration and backward-induction (retrograde analysis) 
slightly improves the time complexity of the previous one. 
Another approach is to analyze games with few cycles, as it is known that 
games without cycles can be solved in polynomial time (Auger et al. \cite{acs}).
Nevertheless, the existence of polynomial time
algorithms for SSGs is still a challenging open question.


In the present paper we propose a new algorithm to find the value and an optimal strategy for the GPG.
This algorithm differs from the usual iterative (in value or in policy) ones and finds the solution directly by a backward-induction.
The number of operations performed by our algorithm is polynomial in the target score $N$ of the game
and also in the number $\mathbf{s}$ of states of the game.
(A more detailed analysis of complexity would require to consider rational probabilities.)
It should be noticed that the GPG we analyze, considered as a SSG,
has a large proportion of random nodes, 
and also a large amount of cycles,
so the general algorithms proposed when this quantities are small are not of use.
In this way the algorithm we propose answers the complexity question for SSG
for our particular class of stochastic games.


The rest of the paper is as follows. 
In section \ref{section:competitive} we formulate the mathematical model of the game and the corresponding Bellman equations. 
In section \ref{section:main} we present our main result: the algorithm and its complexity analysis.
In section \ref{section:examples} we show some numerical results for the pig and piglet games.
A brief conclusion is the content of section \ref{section:conclusion}.

\section{The mathematical model}\label{section:competitive}
 
In this section we formulate the elements of the competitive Markov decision process we consider to model the GPG
and write the corresponding Bellman equations.
Following \cite{Filar} we define: 
the state space; 
the set of available actions each player has at a given state; 
and, given the pair (state, action), 
the reward of the players and the probability distribution for the next state.

\subsection{States}

The states of the game at a given moment is described by the accumulated score of each player
and the turn score of the player who is rolling the die.
Instead of considering the accumulated score of each player, we consider the remaining points to win, 
which contain the same information and are more convenient for our analysis. 
Based on these facts we consider states of the form $(a,b,\tau,j)$, 
where $a$ and $b$ 
($0< a,b\leq N$) are the respective amounts of points that player one and player two need to win the game, 
$\tau$ is the turn score of player $j$, 
while $j=1,2$ indicates which player is rolling the die. 
If $j=1$ then $0\leq\tau\leq a+n-1$, while if $j=2$ we have $0\leq \tau\leq b+n-1$. 
It is also necessary to consider a final absorbing state \texttt{GO}\ (game over).
As the game is symmetric, we usually consider states with $j=1$, 
and omit this information when it is not strictly necessary, i.e. we identify $(a,b,\tau)=(a,b,\tau,1)$.
Observe finally that, for large $N$, we have $\mathbf{s}=O(N^3)$, where $\mathbf{s}$ denotes the number of states of the game. 
 
\subsection{Actions}\label{subsection:4.3}

As we mentioned, we analyze the states with $j=1$, in which player two has no action to choose.
When a turn begins, i.e. at state $(a,b,0)$, player one has only one possible action: to \emph{roll} the die. 
Within the turn, i.e. when $0<\tau<a$ he can \emph{roll} or \emph{hold}. 
In states with $\tau>a$ the only possible action is to \emph{hold}. 
In the special state \texttt{GO}\ there is no action to choose.

\subsection{Probability transitions}\label{subsection:prob}

We define the probability distribution for the next state depending on the present state and action taken by the player. 
Again we describe just player one's states.
From a state $(a,b,\tau)$, with $0<\tau<a$, if the player decides to hold,
the game moves to state $(a-\tau,b,0,2)$ with probability one. 
From a state $(a,b,\tau)$, with $0\leq \tau<a$ if the player rolls, 
the following state will be $(a,b,0,2)$ with probability $p_0$, and $(a,b,\tau+i,1)$ with probability $p_i$ for $i=1\ldots n$. 
From a state $(a,b,\tau)$, with $\tau>a$, and also from the state \texttt{GO}, the following state will be \texttt{GO}\ with probability one.

\subsection{Payoffs}

As player one aims to maximize his winning probability, we define the payoff such that the sum of all his payoffs during the game is one if he wins and zero otherwise. 
To this end we define the payoff to be one at states $(a,b,\tau)$, with $\tau>a$, 
when player one decides to hold (only available action). 
In any other case the payoff is zero. 
Since we are considering a zero-sum game, 
we do not need to define the payoff for player two. 
Observe that, with the defined payoff, the expectation of the sum (over all steps of the game) of all the payoffs, 
coincides with the winning probability of player one.

\subsection{Sub-games}\label{subgames}

A relevant feature of the GPG is that,
once a player obtained certain points in his general account it, is not possible to lose them. 
From the side of the competitive Markov decision process, 
this particularity of the game implies that once the process entered in a state $(a,b,\tau,j)$, 
in the subsequent states it will not visit states $(a',b',\tau',j')$ with $a'>a$ or $b'>b$. 
This allow us to define for each $a$ and $b$ a sub-game, 
with space state restricted to $\{(a',b',\tau',j')\colon a'\leq a,\ b'\leq b\}$.

\subsection{Game value and Bellman equations}

Based on general results for competitive Markov decision process we formulate the optimization problem
and write the corresponding Bellman equations that give the value and an optimal strategy of the game.
For general details and basic definitions see Chapter 4 in \cite{Filar}. 
As was discussed above, 
the transient game has a value and a stationary deterministic optimal strategy within the class of behavioral strategies.
The value of the game is then
$$
v=\operatorname{\bf P}(\text{player one wins from state $(N,N,0)$}),
$$
where $\operatorname{\bf P}$ stands for probability, and both players use their optimal strategies in the class of behavioral strategies.
With the same convention, we define values at intermediate states, by
$$
v(a,b,\tau,j)=\operatorname{\bf P}(\text{player one wins from state $(a,b,{\tau},j)$}),
$$
omitting $j$ when $j=1$, and also omitting $\tau$ when $\tau=0$. With this shorter notation $v=v(N,N)$.
Due to the transience of the game (i.e. it ends with probability one) 
and to its symmetry, we have
\begin{multline*}
\operatorname{\bf P}(\text{player one wins from state $(a,b,0,2)$})\\
=1-\operatorname{\bf P}(\text{player two wins from state $(a,b,0,2)$})\\ 
=1-\operatorname{\bf P}(\text{player one wins from state $(b,a,0,1)$}),
\end{multline*}
that in terms of the values of the game means
\begin{equation*}
v(a,b,0,2)=1-v(b,a,0,1)=1-v(b,a).
\end{equation*}

Using the notation
\begin{equation}
v_{roll}(a,b,\tau)=p_0\left(1-v(b,a)\right)+\sum_{i=1}^n p_i v(a,b,\tau+i),\label{eq:taun}
\end{equation}
the Bellman equations of the game can be written in a compact form, only for $j=1$, as
\begin{equation}
v(a,b,\tau)=
\begin{cases}
v_{roll}(a,b,0),&\quad \tau =0,\\
\max\Big\{1-v(b,a-\tau), v_{roll}(a,b,\tau)\Big\},&\quad 0<\tau<a,\\
1,&\quad \tau\geq a.
\label{tt:bellman}
\end{cases}
\end{equation} 

Observe that if we restrict a solution of the Bellman equations of the game to the states considered in one of the sub-games defined in Section \ref{subgames}, we obtain a solution of the Bellman equations of the sub-game. This fact shows that the solution of a sub-game makes part of the solution of the game. 

\section{Main results}\label{section:main}
\begin{theorem}\label{theorem:1}
There exists a finite algorithm that gives the value and an optimal strategy of the generalized pig game with target $N$.
This algorithm requires $O(N^3\log N)$ steps, which in terms of the number $\mathbf{s}$ of states of the game is 
$O(\mathbf{s}\log\mathbf{s})$.
\end{theorem}
The proof of the Theorem consists in the presentation of the algorithm and the posterior analysis
of its complexity. 
Our algorithm is based in the following facts:
\begin{itemize}
\item The states of the sub-game for $a=\alpha$ and $b=\beta$ can be decomposed: 
for $\alpha=1$, into the states of the smaller sub-game for $a=1$ and $b=\beta -1$, plus the states of the form $(1,\beta,\tau,j)$; 
and for $\alpha>1$,  into the states of the smaller sub-game for  $a=\alpha-1$ and $b=\beta$ plus the states of the form $(\alpha,\beta,\tau,j)$. This allows to implement a backward algorithm.
\item
To find the solution of the sub-game for $a,b$, assuming that is already known the value of the smaller sub-game described in the previous item, one should find the value of the game for the remainder states, i.e. the ones of the form $(a,b,\tau,j)$. Taking into account the reduction to states of player one, this means that the unknown values are that of states of the form $(a,b,\tau)$ and $(b,a,\tau)$.
\end{itemize}
These remarks are consequences of the form of the game and the possible transitions of the underlying Markov process that arises once the actions are taken,
and provide a decoupling of the large optimization problem necessary to solve the game,
into a series of smaller problems that are solved sequentially.

From these  observations we conclude that we can solve the equations recursively backwards, beginning
by $a=b=1$, afterwards fixing $b$ and
solving for $a=1,\dots,b$, from $b=1$ to $b=N$. 
This remark also implies that the winning probabilities from a given state, do not depend on the target $N$ of the game,
 but only on $a$ and $b$, the respective points that players one and two need to win the game.
This discussion leads to the general Algorithm \ref{algo}
\subsection{Pseudo-code for solving the game}\label{pseudo}
\begin{algorithm}
\begin{algorithmic}[1] 
\For{\texttt{$b$ from $1$ to $N$}}
	\For{\texttt{$a$ from $1$ to $b$}}
        \State \texttt{Find $v(a,b,\tau)\colon 0\leq \tau <a$ and $v(b,a,\tau)\colon 0\leq \tau <b$}
      \EndFor
\EndFor
\caption{General backward algorithm.}\label{algo}
\end{algorithmic}
\end{algorithm}
Regarding the complexity, as $\mathbf{s}=O(N^3)$, we must verify that the required number of steps is 
$O(N^3\log N)$.
Here we observe that if $c(a,b)$ 
is the the complexity of step 3 (fixed $a,b$),  and if $c(a,b)\leq c(N,N)$,
then the whole algorithm complexity is at most $O(N^2 c(N,N))$. 
\subsection{Solving step  3 for fixed $a,b$}
In order to implement the above algorithm it is necessary to solve step 3  
for fixed $a$ and $b$. 
From \eqref{tt:bellman}, the corresponding Bellman equations for $v(a,b,\tau)$ with $\tau=a-1,a-2,\dots,0,$ are
\begin{align}
 v(a,b,a-i)&=\max\Big\{1-v(b,i), v_{roll}(a,b,a-1)
 \Big\},\quad (i=1,\dots,a-1)\label{eq:a-1}\\
 v(a,b)&=v_{roll}(a,b,0). \label{eq:1}
 \end{align}
 while for $v(b,a,\tau)$, with $\tau=b-1,b-2,\dots,0$, are:
\begin{align}
 v(b,a,b-j)&=\max\Big\{1-v(a,j), v_{roll}(b,a,b-1)
 \Big\},\quad (j=1,\dots,b-1)
 \label{eq:2}\\
 v(b,a)&=v_{roll}(b,a,0).
 \label{last2}
 \end{align}
We assume, in accordance with the Algorithm \ref{algo}, 
that in the previous steps we have already found 
$v(b,1),\dots,v(b,a-1)$ and $v(a,1),\dots,v(a,b-1)$. 
The following result analyzes equations \eqref{eq:a-1}-\eqref{eq:1}, to show that
$v(a,b,a-i)$ depends only on already known values and on $v(b,a)$. 
Similarly, from \eqref{eq:2}-\eqref{last2}, we have that $v(b,a,b-i)$ depends on already known values and on $v(a,b)$. 
For the following result we introduce the Markov chain $\{X_t\colon t=1,2,\dots\}$ of the scores corresponding to one turn of a player who always roll, 
and starts from 
$X_0=0$. 
Denote by $\Theta_z$ the first hitting time of level $z$. 
In particular $\Theta_0$ is the absorption time.  
\begin{proposition}\label{proposition} 
Assuming $v(b,1),\dots,v(b,a-1);v(a,1),\dots,v(a,b-1)$ as known values, 
then $v(a,b,a-i)\ (i=1,\dots,a-1)$ depend only on $v(b,a)$, 
and $v(b,a,b-i)\ (i=1,\dots,b-1)$ depend only on $v(a,b)$. 
Furthermore, the function $f_{a,b,i}\colon[0,1]\to[0,1],\ (i\leq a)$ 
such that $v(a,b,a-i)=f_{a,b,i}(v(b,a))$, satisfies:
\begin{enumerate}
\item[\rm(a)] For $i\leq 0$, we have $f_{a,b,i}(y)=1$ for all $y\in [0,1]$.
\item[\rm(b)] For $i=0,\ldots,a$, function $f_{a,b,i}$ is piecewise linear, convex (so continuous) and non-increasing. 
It satisfies $f_{a,b,i}(0)=1$, $f_{a,b,i}(1)>0$ and it has up to $i$ points of non-differentiability, 
of which $i-1$ are points of non-differentiability of some $f_{a,b,j}$ for $j< i$.
\item
[\rm(c)]  Denoting $f_{a,b}=f_{a,b,0}$, we have $f_{a,b}'(0)=-\operatorname{\bf P}(\Theta_0<\Theta_a)>-1$.
\item
[\rm(d)] $f_{a,b}$ is strictly decreasing.
\end{enumerate} 
\end{proposition}
\begin{proof} 
The proof of \rm(a) is by induction on $i$.
First observe that \eqref{eq:a-1} can be rewritten as $v(a,b,a-1)=\max\{1-v(b,1),1-p_0 v(b,a)\}$. Then we have 
$$f_{a,b,1}(y)=\max\{1-v(b,1),1-p_0 y\}.$$ 
As $v(b,1)$ is known from a previous step of the algorithm, we have as a function of $y$ is piecewise linear, continuous, non-increasing, with up to one point of non-differentiability, and satisfies $f_{a,b,1}(0)=1$ and $f_{a,b,1}(1)> 0$. 
Assuming the statement is valid for $j=1,\ldots, i-1$, where $a>i>1$, one just need to observe that $v(a,b,a-i)$ is the maximum between $1-v(b,i)$, which is constant in $[0,1]$, and 
$$v_{roll}(a,b,a-i)(y)=p_0(1-y)+p_1 f_{a,b,i-1}(y)+\ldots + p_n f_{a,b,i-n}(y),$$
which is a linear combination of piecewise linear, continuous functions, that satisfy $v_{roll}(a,b,a-i)(0)=1$. 
This shows that $f_{a,b,i}$ is also piecewise linear and continuous, and its possible points of non-differentiability are the ones from 
$f_{a,b,i-j}$, $j=1,\ldots, n$ plus the point $y$ in which $v_{roll}(a,b,a-i)(y)=1-v(b,i)$.
In order to verify (b) we differentiate at the point $y=0$ the equations in \eqref{tt:bellman}, when $\tau<a$, 
applying the chain rule. 
When $y=v(b,a)=0$ the maximum in \eqref{tt:bellman} is the second expression,
and the derivative is: 
\begin{align*}
 v(a,b,\tau)'(0)&=-p_0+\sum_{i=1}^n p_iv(a,b,\tau+i)'(0).
 \end{align*}
 In these equations, whenever 
 $\tau+i\geq a$ 
 we have 
 $v(a,b,\tau+i)'(0)=0$. 
 On the other hand,
 due to the Markov property, for $a\geq 2$ we have that
 \begin{multline*}
 \operatorname{\bf P}(\Theta_0<\Theta_{a-\tau})=
 p_0
 +
 \sum_{i=1}^n p_i
 \operatorname{\bf P}\left(\Theta_0<\Theta_{a-(\tau+i)}\right),\quad
 (\tau=0,\dots,a-1).
 \end{multline*}
We conclude that the sequence of values $\{-v(a,b,\tau)'(0)\colon\tau=0,\dots,a-1\}$ 
and the sequence of probabilities
 $\left\{\operatorname{\bf P}\left(\Theta_0<\Theta_{\tau}\right)\colon\tau=0,\dots,{a-1}\right\}$ 
 satisfy the same recurrence
 relations, with the same initial condition 
 $-v(a,b,a-1)'(0)=\operatorname{\bf P}(\Theta_0<\Theta_1)=p_0$.
 This proves (c) concluding the proof of the proposition.
\end{proof}
 
The previous result shows that, to solve solve step  3, it is enough to know the functions $f_{a,b,i}$, with $i=1,\ldots,a$ and $f_{b,a,i}$, with $i=1,\ldots,b$ and to find $x=v(a,b)$ and $y=v(b,a)$ that solve the system of two equations with two unknowns: 
\begin{equation}\label{eq:s}
\begin{cases}
x=f_{a,b}(y),\\
y=f_{b,a}(x),\\
\end{cases}
\end{equation}
In order to solve step 3 we have Algorithm \ref{algo3}.
\begin{algorithm}
\begin{algorithmic}[1] 
\For{\texttt{$i$ from $1$ to $a$}}
	\State{\texttt{Find the points defining $f_{a,b,i}$}}
\EndFor
	\For{\texttt{$i$ from $1$ to $b$}}
	\State{\texttt{Find the points defining $f_{b,a,i}$}}
      \EndFor
        \State \texttt{Find $x$ and $y$ that solve system \eqref{eq:s}}      
\For{\texttt{$i$ from $1$ to $a-1$}}
	\State{\texttt{compute v(a,b,a-i)}}
\EndFor
\For{\texttt{$i$ from $1$ to $b-1$}}
	\State{\texttt{compute v(b,a,b-i)}}
\EndFor
\end{algorithmic}
\caption{Solving step 3 for fixed $a$, $b$.}\label{algo3}
\end{algorithm}

\begin{proposition}\label{complexity}
To solve Algorithm \ref{algo3} 
we use $c(a,b)\leq O(N\log N)$ steps.
\end{proposition}
\begin{proof} 
Step 3 in the algorithm may be decomposed into three sub-steps: 
(i) the computation of the functions $f_{a,b}$ and $f_{b,a}$; 
(ii) the computation of $x=v(a,b)$ and $y=v(b,a)$ as the solution of the system \eqref{eq:s}; 
and 
(iii) the computation of $v(a,b,i)$ for $i=1,\ldots,a-1$ and $v(b,a,i)$ for $i=1,\ldots,b-1$. 
We analyze the complexity of each of these steps.

Regarding (i), as follows from Bellman equations \eqref{eq:a-1} to \eqref{eq:1}, the function $f_{a,b}$
is the maximum of $a$ linear functions, so its determination is equivalent to the determination of the
common intersection of $a$ half planes. This requires $O(a\log a)$ steps, as follows from Corollary 4.4 in \cite{cgaa}.
We obtain in particular the chain of points that determine this region.

In what respects (ii), based on assertions (a) and (b) in Proposition \ref{proposition}, 
we know that the solution of the non-linear system above is unique, 
being the intersection of two polygonal curves (see Figure \ref{fafb}).
The number of steps necessary to find the intersection of two polygonal lines, each one given as a chain of points
is $O((a+b)\log(a+b))$ if $a$ and $b$ are the respective number of points of the lines (see Corollary 2.7, page 40
in \cite{cgaa}). 

Regarding step (iii), knowning $v(b,a)$, the determination of each value $v(a,b,i)$ demands a substitution,
beginning by $v(a,b,a-1)$, requiring then $O(a)$ steps.

The most demanding step is (ii), that requires a number of steps bounded by $O(N\log N)$,
concluding the proof of Proposition \ref{complexity}. 
\end{proof}
\begin{figure}[h]
\centering
\includegraphics[width=7cm]{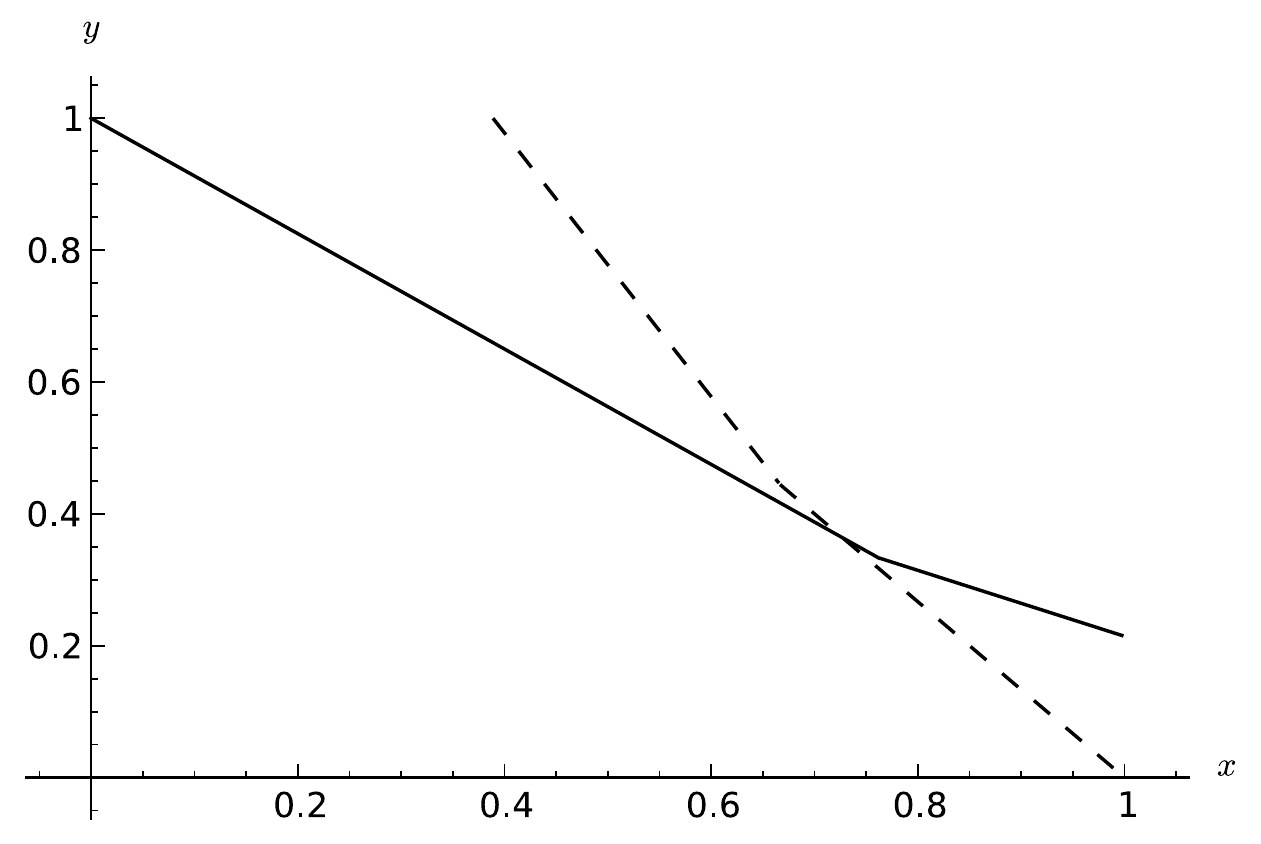}
\caption{Function $y=f_{b,a}(x)$ (solid line) intersects $x=f_{a,b}(y)$ (dashed line) at the solution $x=v(a,b)$, $y=v(b,a)$ in one instance of the Piglet game.}
\label{fafb}
\end{figure}
\begin{proof}[Proof of Theorem \ref{theorem:1}]
We finally observe that, once $v(a,b)$ and $v(b,a)$ are obtained, an optimal strategy is computed checking 
where the maximum is attained at each equation:
if the maximum is
$v_{roll}$ player one has to roll, otherwise he has to hold.
This concludes the proof of the Theorem.
\end{proof}

\begin{remark}
A practical alternative to perform step 3 is to proceed by iteration, 
finding the fixed point of the equation
$$
x=f_{a,b}\left(f_{b,a}(x)\right).
$$
Considering $g(x)=f_a(f_b(x))$, it can be seen, as a consequence of (a) and (c) in Proposition \ref{proposition}, 
that $-1<g'(x)<0$, ensuring that the fixed point of $g$ can be found iteratively.
The values of $x$ and $y$ replaced in equations 
\eqref{eq:a-1} to \eqref{last2}, 
complete the solution of the step 3 of our algorithm, 
and give the value of the game.
\end{remark}
\begin{remark}
The solution of the nonlinear system can also be found
as the solution of the Linear Programming problem:
$$
\left\{
\begin{array}{lll}
\min& x+y,&\\
\text{subject to:}& x\geq f_{a,b}(y),& y\geq f_{b,a}(x).\\
\end{array}
\right.
$$
as can be seen from Figure \ref{fafb}, and results from Proposition \ref{proposition}.
\end{remark}

\section{Examples}\label{section:examples}
\subsection{The pig game}
In Table \ref{table:3} we present the values of the Pig game for different target values. 
\begin{table}
\begin{center}
\caption{Pig Game with different targets}
\begin{threeparttable}
\begin{tabular}{cc}
\hline\noalign{\smallskip}
Target of the game ($N$)& value of the game $v(N,N)$ \\
\hline\noalign{\smallskip}
10       &   0.70942388        \\
50 &      0.54615051     \\
100 &   0.53059207\tnote{1}        \\
200 & 0.52152913\\
500&     0.51362019      \\
1000&    0.50963900       \\ 
\hline\noalign{\smallskip}
\end{tabular}
\begin{tablenotes} 
\item[1] Obtained by Neller and Presser \cite{Neller}
\end{tablenotes}
\label{table:3}
\end{threeparttable}
\end{center}
\end{table} 
Observe that the player who starts rolling has some advantage over his opponent, 
being this advantage less significant as the target increases.
For all target values the value of the game is larger than $0.5$ (as the game is symmetric)
and decreases to $0.5$. 

\subsection{The piglet game}
It is instructive to analyze the  Piglet game as
the function $f_{a,b}$ of Proposition \ref{proposition} can be written explicitly, 
As the algorithm in subsection \ref{pseudo} solves the problem, 
below we explain how to solve step 3.

\subsubsection{The piglet game for fixed $a\leq b$.}
The unknowns are denoted by 
$x_{\tau}=v(a,b,\tau)$ for $\tau=0,\dots,a-1$ 
and 
$y_{\sigma}=v(b,a,\sigma)$ for $\sigma=0,\dots,b-1$.  
The knowns are denoted by 
$\bar{x}_{\tau}=1-v(b,a-\tau,0)$ and $\bar{y}_{\sigma}=1-v(a,b-\sigma,0)$ for the same range in $\tau$ and $\sigma$.
In case $a=b$, as $v(a,b,\tau)=v(b,a,\tau)$ we have $a$ knowns and $a$ unknowns. 
The Bellman equations are
\begin{align*}
x_0&={1-y_0+x_1\over 2}, 					&y_0&={1-x_0+y_1\over 2},\\
x_{1}&=\max\left\{\bar{x}_{1},{1-y_0+x_{2}\over 2}\right\}, &y_{1}&=\max\left\{\bar{y}_{1},{1-x_0+y_{2}\over 2}\right\},\\
x_{2}&=\max\left\{\bar{x}_{2},{1-y_0+x_{3}\over 2}\right\},&y_{2}&=\max\left\{\bar{y}_{2},{1-x_0+y_{3}\over 2}\right\},\\
&\ \vdots &&\vdots\\
x_{a-2}&=\max\left\{\bar{x}_{a-2},{1-y_0+x_{a-1}\over 2}\right\}, &y_{b-2}&=\max\left\{\bar{y}_{b-2},{1-x_0+y_{b-1}\over 2}\right\},\\
x_{a-1}&=\max\left\{\bar{x}_{a-1},{1-y_0+1\over 2}\right\}, &y_{b-1}&=\max\left\{\bar{y}_{b-1},{1-x_0+1\over 2}\right\},
\end{align*}
that, after recursive substitution,
where we denoted $\bar{x}_{a}=\bar{y}_{b}=1$,
give the following nonlinear system for $x_0,y_0$:
\begin{align*}
x_0=&\max_{1\leq \tau\leq a}\left\{\frac1{2^\tau}\left[(2^{\tau}-1)(1-y_0)+\bar{x}_{\tau}\right]\right\},\\
y_0=&\max_{1\leq \sigma\leq b}\left\{\frac1{2^\sigma}\left[(2^{\sigma}-1)(1-x_0)+\bar{y}_{\sigma}\right]\right\}.
\end{align*}
In these equations, 
the properties (a), (b) and (c) of Proposition \ref{proposition} are verified directly.
These equations allow to find the exact solution of the game 
(see Table \ref{table:piglet}).
\begin{table}
\begin{center}
\caption{\label{table:piglet} Values of $v(a,b)$ for the piglet game with $N=3$.}
\begin{tabular}{ccc|c}
$3$&$2$&$1$ & $b\  /  a$\\
\hline\noalign{\smallskip}

\hline
${ 2/9}$&${ 2/5}$&${ 2/3}$&\quad 1\\

${4/11}$&${ 4/7}$&${ 4/5}$&\quad 2\\

$6/11$&${ 8/11}$&${ 8/9}$&\quad 3\\

\end{tabular}
\end{center}
\end{table}

\section{Conclusions}\label{section:conclusion}
In this paper we present an exact algorithm that solves a generalized version of the Pig dice game. 
The fact that the value is found exactly is an important advantage in comparison with other algorithms, which use value iteration or policy iteration.
The algorithm is polynomial time, more precisely requires $O(\mathbf{s}\log\mathbf{s})$ steps (where $\mathbf{s}$ is the number of states of the game). 
It provides an answer in this particular
class of games to the question of the existence of polynomial time algorithms to solve 
Simple stochastic games.


\end{document}